\documentclass[english]{amsart}

\usepackage{esint}
\usepackage[svgnames]{xcolor} 
\usepackage{color}
\usepackage[colorlinks,citecolor=red,pagebackref,hypertexnames=false,breaklinks]{hyperref}
\usepackage{pgf,tikz}
\usepackage{pdfsync}

\usepackage{dsfont}
\usepackage{url}
\usepackage[utf8]{inputenc}
\usepackage[T1]{fontenc}
\usepackage{lmodern}
\usepackage{babel}
\usepackage{mathtools}  
\usepackage{amssymb}
\usepackage{lipsum}
\usepackage{mathrsfs}
\usepackage{color}
\usepackage[skip=2.2pt plus 1pt, indent=12pt]{parskip}
\usepackage{stmaryrd}
\usepackage{soul}

\newtheorem{theorem}{Theorem}[section]
\newtheorem{proposition}{Proposition}[section]
\newtheorem{lemma}{Lemma}[section]

\newtheorem{corollary}{Corollary}[section]

\newtheorem{remark}{Remark}[section]
\newtheorem{example}{Example}[section]

\usepackage{cite}

\numberwithin{equation}{section}

\title[Heat Equations in Spectral Barron Spaces]{Heat Equations in Spectral Barron Spaces}

\author{Mourad Choulli}
\address{Universit\'{e} de Lorraine, 34 cours L\'{e}opold, 54052 Nancy cedex, France}
\email{mourad.choulli@univ-lorraine.fr}

\author{Shuai Lu}
\address{School of Mathematics Sciences, SKLCAM and LMNS, Fudan University, 220 Handan Road, Shanghai 200433, China}
\email{slu@fudan.edu.cn}

\author{Hiroshi Takase}
\address{Department of Mathematics, Okayama University, 3-1-1 Tsushima-naka, Kita-ku, Okayama 700-8530, Japan}
\email{takase@math.okayama-u.ac.jp}

\thanks{This work was supported by National Key Research and Development Programs of China (No. 2023YFA1009103), JSPS KAKENHI Grant Numbers JP25K17280, JP23KK0049, NSFC (No. 92570106) and Science and Technology Commission of Shanghai Municipality (No. 23JC1400501).}
\date{}

\begin{document}

\begin{abstract}
Spectral Barron spaces, characterized by an \(L^1\)-based Fourier-Lebesgue norm, have earned significant attention in approximation theory due to their remarkable capacity to represent functions via shallow neural networks with controlled complexity. Meanwhile, recent theoretical advances have firmly established an intrinsic and profound connection between these function spaces and the regularity theory of elliptic partial differential equations. Building upon this foundational interplay, the present work undertakes a systematic and comprehensive investigation into the well-posedness of heat equations formulated within the spectral Barron spaces framework. Specifically, we rigorously establish the core aspects of well-posedness, including the existence, uniqueness, and stability of solutions, under suitable assumptions on the source terms and conductivity coefficients. We also investigate a typical parabolic inverse problem, namely the backward heat equation, for which we derive a logarithmic conditional stability estimate. To the best of our knowledge, this constitutes the first stability estimate for inverse problems within the spectral Barron space setting. Moreover, we extend our analytical results to address the more intricate setting of time-fractional heat equations, which govern anomalous diffusion phenomena and introduce nonlocal temporal memory effects. In this extended context, we provide a characterization of the corresponding heat kernels, deriving decay estimates, regularity properties, thereby enriching the theoretical landscape of evolutionary PDEs within the spectral Barron spaces setting.
\end{abstract}



\maketitle

\tableofcontents

\section{Introduction}\label{s1}
The interplay between function spaces and partial differential equations (PDEs) lies at the heart of modern analysis. Classical frameworks such as Sobolev, Besov, and H\"older spaces have long served as the standard setting for establishing well‑ posedness, regularity, and numerical approximation of solutions to a wide range of PDE problems, c.f. \cite{Ev,GT,Gr,QV}. Each of these spaces is tailored to specific types of regularity, capturing differentiability in the \(L^p\) sense, fractional smoothness, or uniform continuity, and each has its own strengths and limitations. For instance, Sobolev spaces are ideal for energy estimates and variational formulations, while Besov spaces excel in characterizing nonlinear approximation and boundary trace phenomena. However, with the advent of neural networks and their remarkable success in solving high‑dimensional PDEs, there has been a growing need for function spaces that are not only mathematically tractable but also aligned with the approximation capabilities of neural networks. This demand has propelled spectral Barron spaces to the forefront of contemporary research.

Spectral Barron spaces, denoted by \(B^s\) (or \(B^s(\mathbb{R}^n)\)) throughout this paper, are defined via a weighted Fourier–Lebesgue norm; a rigorous definition will be provided in Section \ref{s2}. Their origin lies in the seminal work of Barron \cite{Ba} on the approximation of functions by shallow neural networks. Barron proved that any function with a finite spectral Barron norm, for instance, the integral of \(|\xi|\,|\widehat{f}(\xi)|\) over \(\mathbb{R}^n\), can be uniformly approximated by a shallow neural network with a sigmoid activation function and \(m\) neurons, achieving an approximation error of order \(O\left(\frac{1}{\sqrt{m}}\right)\). This result was later refined and extended to ReLU‑type activation functions, as demonstrated in \cite{KB_IEEE}. These findings have profound implications for the theory of deep learning, as they provide deterministic guarantees on generalization error without curse of dimensionality and illuminate why overparameterized networks can still generalize effectively when the target function belongs to a spectral Barron space class, i.e \cite{LiaoMing}.

Beyond approximation theory, by introducing different weights i.e. $(1+|\xi|)^s$ or $(\sqrt{1+|\xi|^2})^{s}$, spectral Barron spaces have recently emerged as a powerful tool in the analysis of elliptic PDEs. A series of works (see, e.g., \cite{CLLZ,CLT,LKH,Xu2020}) has demonstrated that solutions of second‑order elliptic equations with sufficiently regular coefficients, posed on the whole space or subject to suitable boundary conditions, automatically belong to \(B^s\) for certain \(s\). This is a non‑trivial fact: elliptic regularity in classical Sobolev spaces only yields \(H^{m+2}\) regularity, whereas the stronger integrability of the Fourier transform imposed by the \(B^s\) norm places a more stringent constraint on the high‑frequency decay of the solution. The key observation is that the Green's function (or fundamental solution) of a uniformly elliptic operator has a Fourier symbol that behaves like \(|\xi|^{-2}\) at infinity. When this symbol is multiplied by the Fourier transform of a source term in \(B^s\), the resulting product remains integrable in the weighted \(L^1\) sense. In fact, it precisely satisfies the weighted integrability condition for \(B^{s+2}\), thanks to the convolution structure and the extra decay of the source term. This yields a clean frequency‑domain characterization of the solution and leads to sharp a priori estimates that are not readily obtainable in Sobolev spaces. Moreover, the spectral Barron norm controls both the solution and its derivatives up to a certain order, thereby enabling the derivation of stability bounds for numerical schemes based on neural networks, c.f. \cite{LKH,Xu2020}.

Despite this success for elliptic problems, the study of evolutionary PDEs within the spectral Barron framework remains largely unexplored. 
Only a few references are available on this topic. We therefore refer to the very recent work \cite{CLSSS}, which introduces anisotropic spectral Barron spaces via a global space-time Fourier extension. These spaces are specifically designed for space-fractional parabolic equations with lower-order terms, for which dimension-independent regularity results and neural network approximation rates are established. On the other hand, if the time variable is treated classically while the spatial variable remains within the spectral Barron spaces framework, no results have been reported to the best of our knowledge. As the simplest and most fundamental parabolic equation, the heat equation serves as a natural starting point. 
The solution to a standard heat equation is given by the convolution with the Gaussian heat kernel \(G_t(x) = (4\pi t)^{-n/2} e^{-\frac{|x|^2}{4t}}\), whose Fourier transform is \(\widehat{G_t}(\xi) = e^{-t|\xi|^2}\). This exponential multiplier is markedly different from the algebraic symbol of elliptic operators. In the Fourier domain, the heat equation corresponds to multiplying the initial data and the time‑integrated source term by the factor \(e^{-t|\xi|^2}\). A question then arises: does this multiplier preserve the spectral Barron norm? Moreover, the heat equation introduces a time variable that changes the regularity profile: for positive time, the solution becomes infinitely smooth in the classical sense, but the spectral Barron norm may decay at a certain rate. Characterizing this decay rate is essential for understanding the long‑time behavior and for deriving error bounds when approximating the solution by neural networks at each time step. To further broaden the scope, we will consider the time‑fractional heat equation, which models anomalous diffusion processes, where the mean‑squared displacement grows sublinearly in time. Such a phenomenon has been observed in many physical, biological, and financial systems. In particular, the fractional derivative introduces a memory kernel that renders the evolution nonlocal in time; the solution can no longer be expressed by a simple exponential semigroup but involves other structures. It is worth noting that none of these aspects have been systematically addressed in the existing literature on spectral Barron spaces. 

Motivated by the gaps identified above, the present paper focuses on the (fractional) heat equation within the spectral Barron space setting. As highlighted in \cite{CLT,LuM}, the Laplacian is a sectorial operator on the spectral Barron space \(B^0\); it therefore generates an analytic semigroup. Leveraging this property, we investigate the existence and uniqueness of solutions to heat equations in \(B^0\). We also investigate a typical parabolic inverse problem, namely the backward heat equation, for which we derive a logarithmic conditional stability estimate. Our results rely on the well-established theory of evolution equations associated with sectorial operators. For clarity, we restrict ourselves to a few representative examples, although many other results could naturally be obtained. Furthermore, by adapting certain proofs from \cite{CY}, originally developed for \(m\)-dissipative operators, we establish existence and uniqueness results for time-fractional heat equations in \(B^0\). We also propose an extension to heat equations with space-dependent coefficients. Using a density argument, we show that the Gaussian heat kernel in the \(L^2\) setting is also the heat kernel of the Laplacian in \(B^0\); interestingly, the proofs of certain properties of this heat kernel in \(B^0\) turn out to be less technical than their counterparts in \(L^2\).

Our work complements several lines of research. First, it extends the elliptic theory of Barron spaces (e.g., \cite{CLLZ,CLT,LKH,Xu2020}) to the parabolic realm. Second, it contributes to the extensive literature on fractional diffusion by providing a new functional-analytic framework that is particularly well suited to neural network approximation. In particular, the obtained results builds a bridge between rigorous PDE theory and the rapidly evolving field of scientific machine learning, offering a solid foundation for algorithms that employ deep networks as ansatz for time‑dependent problems. Although there have been numerical studies using neural networks for heat equations, most lack rigorous error bounds in frequency‑based norms; our results initialize to address this gap.

The remainder of this article is organized as follows. Section \ref{s2} is devoted to preliminaries. The main properties of spectral Barron spaces established in \cite{CLT} are recalled in Subsection \ref{sb2.1}. In Subsection \ref{sb2.2}, we demonstrate that the Laplacian is a sectorial operator on \(B^0\), and in Subsection \ref{sb2.3}, we examine the intermediate spaces associated with the Laplacian on \(B^0\) and present an embedding result. Spaces of vector-valued continuous functions of a single variable, used later in this article, are introduced in Subsection \ref{sb2.4}. Section \ref{s3} deals with the existence and uniqueness of solutions for heat equations, including semilinear ones. In this section, we also study the Gaussian heat kernel as the kernel of the heat equation on \(B^0\) and the backward heat equation. Cauchy problems for linear time-fractional heat equations are addressed in Section \ref{s4}. Finally, in Section \ref{s5}, we extend certain results obtained for the heat equation in \(B^0\) to heat equations with space-dependent coefficients, and we also discuss their heat kernels.

\section{Preliminaries}\label{s2}

In this section, we collect the definitions and necessary preliminaries for the present work. Following \cite{CLT}, we will use the shorthand notations \(\mathscr{S}\), \(B^s\), \(L^p\), etc., in place of the more cumbersome \(\mathscr{S}(\mathbb{R}^n)\), \(B^s(\mathbb{R}^n)\), \(L^p(\mathbb{R}^n)\), and so on. We recall that \(\mathscr{S}\) denotes the Schwartz space on \(\mathbb{R}^n\), while \(B^s\) stands for the spectral Barron spaces employed throughout this article; their precise definitions are given in the following subsection. We also adopt the notation \(C_b^k\), \(k\in \mathbb{N}\cup\{0\}\), for the space of bounded functions on \(\mathbb{R}^n\) that are \(k\)-times continuously differentiable and have bounded derivatives up to order \(k\). This space is equipped with the norm
\[
\|f\|_{C_b^k}:=\sup\left\{\|\partial^\alpha f\|_{L^\infty};\; \alpha=(\alpha_1,\ldots ,\alpha_n)\in \mathbb{N}^n\cup\{0\},\; |\alpha|:=\alpha_1+\ldots +\alpha_n\le k\right\}.
\]
For \(0<\beta \le 1\), we define the Hölder seminorm \([\cdot]_\beta\) by
\[
[f]_\beta:=\sup\left\{\frac{|f(x)-f(y)|}{|x-y|^\beta};\; x,y\in \mathbb{R}^n,\; x\ne y\right\}.
\]
For \(k\in \mathbb{N}\cup\{0\}\), let
\[
C_b^{k,\beta}:=\left\{f\in C_b^k;\; [\partial^\alpha f]_\beta<\infty,\; \alpha\in \mathbb{N}^n\cup\{0\},\; |\alpha|=k\right\},
\]
and endow it with the norm
\[
\|f\|_{C_b^{k,\beta}}:=\|f\|_{C_b^k}+\sup_{|\alpha|=k}[\partial^\alpha f]_\beta.
\]
It is well known that both \(C_b^k\) and \(C_b^{k,\beta}\), equipped with their respective norms, are Banach spaces.

\subsection{Spectral Barron spaces $B^s$, $s\geq 0$}\label{sb2.1}
In this subsection, we collect the properties of spectral Barron spaces that will be used in the sequel. To this end, recall that, for \(s \ge 0\), the spectral Barron space \(B^s\) is defined by
\[
B^s := \{f\in C_b^0;\;\langle\xi\rangle^s\hat{f}\in L^1\},
\]
where \(\langle\xi\rangle:=\sqrt{1 + |\xi|^2}\). Equipped with the norm
\[
\|f\|_{B^s}:=\|\langle\xi\rangle^s\hat{f}\|_{L^1},\quad f\in B^s,
\]
\(B^s\) is a Banach space. Moreover, if \(0 \leq s \leq t\), we have the continuous embedding \(B^t \hookrightarrow B^s\), together with the norm inequality
\[
\|f\|_{B^s} \leq \|f\|_{B^t},\quad f \in B^t.
\]

For \(0\le r<t\), \(1\le p\le \infty\), and \(0<\theta\le 1\), we recall the definition of the real interpolation space \((B^r, B^t)_{\theta, p}\) as follows:
\[
(B^r, B^t)_{\theta, p} := \left\{ f \in B^r ; \; \rho^{-\theta-\frac{1}{p}} K(\cdot, f) \in L^p\left((0, \infty)\right) \right\},
\]
with the usual convention that \(\frac{1}{\infty}=0\), where
\[
K(\rho, f) := \inf \left\{ \|g\|_{B^r} + \rho \|h\|_{B^t}; \; g \in B^r, \; h \in B^t \; \text{such that} \; g + h = f \right\}.
\]
The Banach space \((B^r, B^t)_{\theta, p}\) is endowed with its natural norm:
\[
\|f\|_{(B^r, B^t)_{\theta, p}} := \left\| \rho^{-\theta-\frac{1}{p}} K(\cdot, f) \right\|_{L^p\left((0, \infty)\right)},\quad f\in (B^r, B^t)_{\theta, p}.
\]

The main properties of the spectral Barron spaces \(B^s\) (\(s\ge 0\)), which we established in \cite{CLT}, are summarised in the following theorem.

\begin{theorem}\cite{CLT}\label{thm1}
$\mathrm{(i)}$ Let \(0 \le r \le s\le t\), and \(\alpha \in [0, 1]\) be such that \(s = \alpha t + (1 - \alpha) r\). Then the following inequality holds:
\begin{equation}\label{ii}
\|f\|_{B^s} \leq \|f\|_{B^r}^{1-\alpha} \|f\|_{B^t}^{\alpha}, \quad f \in B^t.
\end{equation}
\\
$\mathrm{(ii)}$ The space \(\mathscr{S}\) is dense in \(B^s\) for every \(s \geq 0\).
\\
$\mathrm{(iii)}$ For \(s \geq 0\) and \(t > s + \frac{n}{2}\), it holds \(\langle \xi \rangle^{s - t} \in L^2\), the embedding \(H^t \hookrightarrow B^s\), and the estimate
\[
\|f\|_{B^s} \leq \|\langle \xi \rangle^{s - t}\|_{L^2} \|f\|_{H^t}, \quad f \in H^t.
\]
$\mathrm{(iv)}$ For every \(s \geq 0\), if \(f, g \in B^s\), then \(fg\in B^s\) and the product estimate holds:
\begin{equation}\label{prod}
\|fg\|_{B^s} \leq 2^{\frac{s}{2}} (2\pi)^{-n} \|g\|_{B^s} \|f\|_{B^s}.
\end{equation}
\\
$\mathrm{(v)}$ For all \(s\ge 0\), it holds the continuous embedding \(B^s\hookrightarrow C_b^{\lfloor s\rfloor}\), where \(\lfloor s\rfloor\) denotes the integer part of \(s\).
\\
$\mathrm{(vi)}$ For every integer \(k \geq 0\), and for \(0 < \theta < 1\) and \(0 < \gamma < \theta\), the embedding \(B^{k+\theta} \hookrightarrow C_b^{k,\gamma}\) holds.
\\
$\mathrm{(vii)}$ If \(0 \leq r < s < t\), \(\theta = \frac{s - r}{t - r}\), and \(0 < \tilde{\theta} < \theta\), then
\[
(B^r, B^t)_{\theta, 1} \hookrightarrow B^s \hookrightarrow (B^r, B^t)_{\tilde{\theta}, 1}.
\]
\\
$\mathrm{(viii)}$ For every integer \(k \geq 0\), it holds \(W^{n+k,1} \hookrightarrow B^k\).
\\
$\mathrm{(ix)}$ For all \(p\ge 1\), integers \(k\ge 0\), and for \(0<\theta<1\) and \(0<\sigma <\frac{\theta}{p}\), it holds that \(B^{k+\theta}\subset W_{\rm{loc}}^{k+\sigma,p}\).
\end{theorem}

For later use, we recall from \cite[Proposition 1.2.3]{Lun} that the following embeddings hold for all \(0<\theta <1\) and all \(1\le p_1\le p_2\le \infty\):
\begin{equation}\label{intr1}
B^2\hookrightarrow (B^0,B^2)_{\theta,p_1}\hookrightarrow (B^0,B^2)_{\theta,p_2}\hookrightarrow(B^0,B^2)_{\theta,\infty}\hookrightarrow B^0.
\end{equation}
Moreover, for all \(0<\theta_1<\theta_2\le 1\), we have
\begin{equation}\label{intr2}
(B^0,B^2)_{\theta_2,\infty}\hookrightarrow (B^0,B^2)_{\theta_1,1}.
\end{equation}
Furthermore, by \cite[Corollary 1.2.7]{Lun}, for every \((\theta,p)\in N\), there exists a constant \(c=c(\theta,p)>0\) such that
\begin{equation}\label{intr3}
\|f\|_{(B^0,B^2)_{\theta,p}}\le c\|f\|_{B^0}^{1-\theta}\|f\|_{B^2}^\theta,\quad f\in B^2.
\end{equation}
Here and throughout, \(N:=(0,1)\times [1,\infty]\).

\begin{remark}\label{rem1}
{\rm
For \(s\ge 0\), we introduce the following spectral Barron space of negative order:
\[
\tilde{B}^{-s}:=\{ f\in \mathscr{S}';\; \langle \xi \rangle^{-s}\mathscr{F}^{-1}f\in L^\infty\}.
\]
This space is endowed with the norm
\[
\|f\|_{\tilde{B}^{-s}}:= \|\langle \xi \rangle^{-s}\mathscr{F}^{-1}f\|_{L^\infty},\quad f\in \tilde{B}^{-s}.
\]
From \cite[Proposition 2.2]{CLT}, it is known that \((B^s)'\), the dual of \(B^s\), can be identified with \(\tilde{B}^{-s}\).

Let \(0\le r<t\), \(1\le p<\infty\), and \(0<\theta<1\). Since \(B^t\cap B^r=B^t\) is dense in \(B^r\), applying \cite[Theorem 1.18]{Lu} yields the duality relation
\[
\left[(B^r,B^t)_{\theta,p}\right]'=(\tilde{B}^{-t},\tilde{B}^{-r})_{\theta,p'},
\]
where \(p'\in (1,\infty]\) is the conjugate exponent of \(p\). This duality result, combined with Theorem \ref{thm1} (vii), implies that for \(0 \leq r < s < t\), \(\theta = \frac{s - r}{t - r}\), and \(0 < \tilde{\theta} < \theta\), we have the embeddings
\[
(\tilde{B}^{-t},\tilde{B}^{-r})_{\tilde{\theta},\infty}\hookrightarrow \tilde{B}^{-s}\hookrightarrow (\tilde{B}^{-t},\tilde{B}^{-r})_{\theta,\infty}.
\]
}
\end{remark}

\subsection{The Laplacian as an operator acting on $B^0$}\label{sb2.2}

Define the unbounded operator \( A: B^0 \rightarrow B^0 \) by
\[
Au := \Delta u, \quad u \in D(A) := B^2,
\]
and recall that the resolvent set of \( A \) is defined by
\[
\rho(A) := \{ \lambda \in \mathbb{C} ; \; \lambda - A : D(A) \rightarrow B^0 \text{ is a bijection and } (\lambda-A)^{-1}\in \mathscr{B}(B^0)\}.
\]
As established in \cite{CLT}, \(A\) is closed and has dense domain.

For a Banach space \(X\) with norm \(\|\cdot\|\), a linear unbounded operator \(\mathbf{A}:D(\mathbf{A})\subset X\rightarrow X\) is said to be sectorial if there exist \(\theta\in (\frac{\pi}{2},\pi)\) and \(\omega\in \mathbb{R}\) such that
\[
\rho(\mathbf{A})\supset S_{\theta,\omega}:=\{z\in \mathbb{C}\setminus\{\omega\};\; |\arg(z-\omega)|<\theta\},
\]
and
\[
\sup\{\|(\lambda -\omega)(\lambda-\mathbf{A})^{-1}\|_{\mathscr{B}(X)};\; \lambda \in S_{\theta,\omega}\}<\infty.
\]
By \cite[Definition 2.0.2]{Lun}, a sectorial operator \(\mathbf{A}:D(\mathbf{A})\subset X\rightarrow X\) generates an analytic semigroup \((e^{t\mathbf{A}})_{t\ge 0}\). This semigroup is strongly continuous whenever \(D(\mathbf{A})\) is dense in \(X\). We shall make use of the following sufficient condition for an unbounded operator \(\mathbf{A}:D(\mathbf{A})\subset X\rightarrow X\) to be sectorial (see, e.g., \cite[Proposition 2.1.11]{Lun}): if there exist \(\omega \in \mathbb{R}\) and \(M>0\) such that \(\{z\in \mathbb{C};\; \Re z\ge \omega\} \subset \rho(\mathbf{A})\) and
\[
\|\lambda (\lambda -\mathbf{A})^{-1}\|_{\mathscr{B}(X)}\le M,\quad \Re \lambda \ge \omega,
\]
where $\mathscr{B}(X)$ denotes the operator norm on $X$, then \(\mathbf{A}\) is sectorial.

We then establish the sectorial property of the Laplacian operator $A$ on the spectral Barron space $B^0$. Similar results are found in \cite[Corollary 3.1]{CLT} and \cite[Proposition 2.1]{LuM}.

\begin{proposition}\label{pro1}
It holds that the set $\{z\in \mathbb{C};\; \Re z>0\}\subset \rho(A)$ and that, for $\Re \lambda >0$, the estimate  
\[
\|\lambda (\lambda-A)^{-1}\|_{\mathscr{B}(B^0)}\le 1
\]  
is valid. 
Therefore, $A$ is sectorial.
\end{proposition}

\begin{proof}
Let $f\in B^0$ and $\lambda \in \mathbb{C}$ such that $\Re \lambda>0$. If $u\in B^2$ is the solution of 
\begin{equation}\label{res}
(\lambda -\Delta )u=f,
\end{equation}
then 
\[
(\lambda +|\xi|^2)\hat{u}=\hat{f}.
\]
Equivalently, we have
\[
\hat{u}=\frac{\hat{f}}{\lambda +|\xi|^2}.
\]
Therefore, \eqref{res} admits a unique solution $u=(\lambda-A)^{-1}f\in B^2=D(A)$ and
\[
\|u\|_{B^0}\le \frac{1}{|\lambda|}\|f\|_{B^0}.
\]
The proof is complete.
\end{proof}

Next, consider the operator $ \mathbb{A} := -A+1$ with domain $ D(\mathbb{A}) = D(A) $. By Proposition \ref{pro1}, $ \mathbb{A} $ is a closed and densely defined operator such that $ (-\infty, 1[\subset \rho(\mathbb{A}) $, and the following estimate holds:  
\[
\sup_{t \geq 0} (1 + t) \| (t + \mathbb{A})^{-1} \|_{\mathscr{B}(B^0)} \le 1.  
\]  
Thus, $ \mathbb{A} $ is a positive operator in the sense of \cite[Definition 4.1]{Lu}. Then the fractional power of $\mathbb{A}$, $\mathbb{A}^\alpha $, $0<\alpha <1$, is well defined and is expressed as
\[
\mathbb{A}^\alpha f = \frac{\sin (\pi \alpha)}{\pi} \mathbb{A} \int_0^\infty \lambda^{-1 + \alpha} (\lambda + \mathbb{A})^{-1} f \, d\lambda, \quad f \in D(\mathbb{A}^\alpha).  
\]  
Furthermore, we have proved in \cite{CLT} that 
\begin{equation}\label{dom}
D(\mathbb{A}^\alpha)=B^{2\alpha},\quad 0\le \alpha \le 1,
\end{equation}
where $\mathbb{A}^0:=I$.

Note that, since $B^2$ is dense in $B^0$, the analytic semigroup generated by $A$, i.e. $e^{tA}$,  is strongly continuous.
We can proceed as in the proof of \cite[Proposition 2.1.1]{Lun} to establish that $e^{-t\mathbb{A}}$  is bounded:
\[
\sup_{t\ge 0}\|e^{-t\mathbb{A}}\|_{\mathscr{B}(B^0)}<\infty.
\]

\subsection{Intermediate spaces}\label{sb2.3}

For each $(\alpha ,p)\in N$, define
\[
D_{\mathbb{A}}(\alpha,p):=\left\{f\in B^0;\; t\mapsto w_f(t):=\|t^{1-\alpha-\frac{1}{p}}\mathbb{A}e^{-t\mathbb{A}}f\|_{B^0}\in L^p((0,1))\right\}. 
\]
The norm on $D_{\mathbb{A}}(\alpha,p)$ is typically given by
\[
\|f\|_{D_{\mathbb{A}}(\alpha,p)}=\|f\|_{B^0}+\|w_f\|_{L^p((0,1))},\quad f\in D_{\mathbb{A}}(\alpha,p).
\]
Alternatively, in view of \cite[Proposition 2.2.4]{Lun}, the following equivalent characterization holds:
\[
D_{\mathbb{A}}(\alpha,p):=\left\{f\in B^0;\; t\mapsto t^{-\alpha-\frac{1}{p}}\|e^{-t\mathbb{A}}f-f\|_{B^0}\in L^p((0,1))\right\}.
\]
The space $D_{\mathbb{A}}(\alpha,p)$ serves as an intermediate space between $B^0$ and $D(\mathbb{A})=B^2$. It was shown in \cite[Proposition 2.2.2]{Lun} that
\begin{equation}\label{is1}
D_{\mathbb{A}}(\alpha,p)=(B^0,B^2)_{\alpha,p}.
\end{equation}

For $0<\alpha<1$, let $J_\alpha$ denote the collection of spaces $X_\alpha$ such that $B^0\hookrightarrow X_\alpha \hookrightarrow B^2$ and, for some constant $c>0$, 
\[
\|f\|_{X_\alpha}\le c\| f\|_{B^0}^{1-\alpha}\|f\|_{B^2}^\alpha ,\quad f\in B^2.
\]
By Theorem \ref{thm1} (i), $B^{2\alpha}\in J_\alpha$. On the other hand, combining \eqref{intr3} and \eqref{is1} yields $D_{\mathbb{A}}(\alpha,p)\in J_\alpha$ for all $(\alpha,p)\in N$. Then the following embedding property holds.

\begin{theorem}\label{thm2}
Let $0<\beta <\alpha<\gamma <1$. Then
\begin{equation}\label{isi}
B^{2\gamma}\hookrightarrow D_{\mathbb{A}}(\alpha,1)\hookrightarrow D_{\mathbb{A}}(\alpha,\infty)\hookrightarrow B^{2\beta} .
\end{equation}
\end{theorem}

\begin{proof}
In what follows, we use the fact that $D(\mathbb{A}^\eta)=B^{2\eta}$ for $0<\eta\le 1$. Since $D_{\mathbb{A}}(\alpha,p)=(B^0,B^2)_{\alpha,p}$ for all $(\alpha,p)\in N$ by \eqref{is1}, it follows from \eqref{intr2} that
\[
D_{\mathbb{A}}(\alpha,\infty)\hookrightarrow (B^0,B^2)_{\beta ,1} .
\]
On the other hand, Theorem \ref{thm1} (vii) yields
$B^{2\gamma}\hookrightarrow (B^0,B^2)_{\alpha,1}=D_{\mathbb{A}}(\alpha,1)$ and   $(B^0,B^2)_{\beta ,1}\hookrightarrow B^{2\beta}$. Hence,
\[
B^{2\gamma}\hookrightarrow D_{\mathbb{A}}(\alpha,1)\hookrightarrow D_{\mathbb{A}}(\alpha,\infty)\hookrightarrow B^{2\beta}.
\]
This completes the proof.
\end{proof}

From \cite[Corollary 2.2.3]{Lun}, we have the following lemma.

\begin{lemma}\label{lem2}
Let $E$ be a Banach space satisfying $B^2\hookrightarrow E\hookrightarrow B^0$. Then, $E\in J_\alpha$ if and only if $D_{\mathbb{A}}(\alpha,1)\hookrightarrow E$.
\end{lemma}

\begin{remark}\label{rem2}
{\rm
In light of \eqref{is1} and Remark \ref{rem1}, we have
\[
\left[D_{\mathbb{A}}(\alpha, p)\right]'=(\tilde{B}^{-2},\tilde{B}^0)_{\alpha,p'},\quad 0<\alpha <1,\; 1\le p<\infty.
\]
}
\end{remark}

\subsection{Some other function spaces}\label{sb2.4}

Let $X$ be an arbitrary Banach space, whose norm is denoted by $\|\cdot \|$, and let $I$ be an interval of $\mathbb{R}$ (possibly unbounded). For $\alpha\in (0,1]$ and $f:I\rightarrow X$, let
\[
[f]_\alpha:=\sup\left\{ \frac{\|f(t)-f(s)\|}{|t-s|^\alpha};\; t,s\in I,\; t\ne s\right\}.
\]
Assume that $\alpha\in (0,1)$. Then, define
\[
C^\alpha (I,X)=\{f\in C_b^0(I,X);\; [f]_\alpha<\infty\},
\]
which is a Banach space for the norm
\[
\|f\|_{C^\alpha (I,X)}=\|f\|_{C_b^0(I,X)}+[f]_\alpha,\quad f\in C^\alpha (I,X).
\]
For all $k\in \mathbb{N}$, let
\[
C^{k+\alpha}(I,X):=\{f\in C_b^k(I,X);\; f^{(k)}\in C^\alpha (I,X)\}.
\]
This space is complete when endowed with its natural norm, given by
\[
\|f\|_{C^{k+\alpha}(I,X)}=\|f\|_{C_b^k(I,X)}+\left[f^{(k)}\right]_\alpha.
\]
For the reader's convenience, we retain the notations from \cite{Lun}, although they differ slightly from those used for scalar functions.

Also, define
\[
\mathrm{Lip}(I,X):=\left\{f\in C^0_b(I,X);\; [f]_1<\infty\right\}.
\]
$\mathrm{Lip}(I,X)$ is a Banach space for the norm
\[
\|f\|_{\mathrm{Lip}(I,X)}:=\|f\|_{C^0_b(I,X)}+[f]_1.
\]

Next, for $\alpha \in (0,2)$ and $f:I\rightarrow X$, let
\[
[\![ f]\!]_\alpha:=\sup \left\{\frac{\left\|f(s)-2f\left(\frac{s+t}{2}\right)+f(s)\right\|}{|t-s|^\alpha};\; s,t\in I,\; t\ne s \right\}
\]
and
\[
\tilde{C}^\alpha (I,X):=\{f\in C_b^0(I,X);\; [\![ f]\!]_\alpha<\infty\}.
\]
Equipped with the norm
\[
\|f\|_{\tilde{C}^\alpha (I,X)}:=\|f\|_{C_b^0(I,X)}+[\![ f]\!]_\alpha,\quad f\in \tilde{C}^\alpha (I,X),
\]
$\tilde{C}^\alpha (I,X)$ is a Banach space.

Note that, by \cite[Proposition 0.2.2]{Lun}, we have
\begin{equation}\label{hoe}
\tilde{C}^\alpha (I,X)=C^\alpha (I,X),\quad 0<\alpha<2,\; \alpha \ne 1.
\end{equation}

The following lemma will be used later. It follows from \cite[Proposition 1.1.5 and Proposition 2.2.12]{Lun}.

\begin{lemma}\label{lem1}
Let $I$ be an interval of $\mathbb{R}$, let $\alpha, \theta\in (0,1)$, and let $X_\alpha \in J_\alpha$. Then
\[
C^\theta (I,B^2)\cap C^{1+\theta}(I,B^0)\hookrightarrow \tilde{C}^{\theta+1-\alpha}(I,X_\alpha).
\]
In particular, for all $p\ge 1$, we have
\[
C^\alpha (I,B^2)\cap C^{1+\alpha}(I,B^0)\hookrightarrow \tilde{C}^1(I,D_{\mathbb{A}}(\alpha,p)).
\]
For the case $p=\infty$, we have
\[
C^\alpha (I,B^2)\cap C^{1+\alpha}(I,B^0)\hookrightarrow \mathrm{Lip}(I,D_{\mathbb{A}}(\alpha,\infty)).
\]
Furthermore, if $f\in C^\alpha (I,B^2)\cap C^{1+\alpha}(I,B^0)$, then $f'(t)\in D_{\mathbb{A}}(\alpha,\infty)$ for all $t\in I$.
\end{lemma}

\section{Heat equations in $B^0$}\label{s3}

Our objective is to study the existence and uniqueness of solutions to the Cauchy problem for the heat equation of the following form:
\begin{equation}\label{heat1}
\left\{
\begin{array}{ll}
\partial_tu(t,x)-\Delta u(t,x) =q(t,x)+a(t,u(t,x))\quad \mbox{in}\; (0,\tau)\times \mathbb{R}^n,
\\
u(0,\cdot)=u_0,
\end{array}
\right.
\end{equation}
where $\tau\in (0,\infty)$. This Cauchy problem encompasses both linear and semilinear cases. Formally, $u$ is a solution to \eqref{heat1} if and only if $v(t,\cdot)=e^{-t}u(t,\cdot)$ is a solution to the following Cauchy problem:
\begin{equation}\label{heat2}
\left\{
\begin{array}{ll}
\partial_tv(t,x)+v(t,x)-\Delta v(t,x)
\\
\hskip 3cm
=e^{-t}q(t,x)+e^{-t}a(t,e^tv(t,x))\quad \mbox{in}\; (0,\tau)\times \mathbb{R}^n,
\\
v(0,\cdot)=u_0.
\end{array}
\right.
\end{equation}

Noting that $u$ and $v$ share the same regularity, it suffices to consider \eqref{heat2}. By identifying $v(t,\cdot)$ with $v(t)$, we are led to study the following abstract Cauchy problem in the space $B^0$:
\begin{equation}\label{aeq0}
\left\{
\begin{array}{ll}
v'(t)+\mathbb{A} v(t) =F(t,v(t)),\quad t\in (0,\tau),
\\
v(0)=u_0.
\end{array}
\right.
\end{equation}

\subsection{Linear heat equations}\label{sb3.1}

Let $0<\tau<\infty$ be fixed arbitrarily. Define the operator $\mathscr{C}$ by
\[
\mathscr{C}: f\in L^\infty ((0,\tau),B^0)\mapsto \mathscr{C}(f)(t)=(e^{-t\mathbb{A}}\ast f)(t):=\int_0^te^{-(t-s)\mathbb{A}}f(s)ds,
\]
where we recall that $e^{-t\mathbb{A}}$ is the semigroup generated by $\mathbb{A}$. For a more extensive discussion on semigroup theory for PDEs, we refer to \cite{Lun,Pazy}.

Applying \cite[Proposition 4.2.1]{Lun} yields the following result.

\begin{proposition}\label{pro2}
For all $\alpha \in (0,1)$, it holds \[\mathscr{C}\in \mathscr{B}(L^\infty((0,\tau),B^0),C^{1-\alpha}([0,\tau],D_{\mathbb{A}}(\alpha,1))).\]
\end{proposition}

Proposition \ref{pro2}, combined with Lemma \ref{lem2}, gives the following corollary.

\begin{corollary}\label{cor1}
For all $\alpha \in (0,1)$ and $X_\alpha \in J_\alpha$, it holds \[\mathscr{C}\in \mathscr{B}(L^\infty((0,\tau),B^0),C^{1-\alpha}([0,\tau],X_\alpha)).\]
In particular, by Theorem \ref{thm1} {\rm (vi)}, 
\[\mathscr{C}\in \mathscr{B}(L^\infty((0,\tau),C_b^0),C^{1-\alpha}([0,\tau],C_b^{\lfloor 2\alpha \rfloor, 2\alpha -\lfloor 2\alpha \rfloor})\;  \mathrm{if}\;  \alpha\ne \frac{1}{2} \]
and
\[\mathscr{C}\in \mathscr{B}(L^\infty(0,\tau),C_b^0),C^{\frac{1}{2}}([0,\tau],C_b^1) .\]
\end{corollary}

The operator $\mathscr{C}$ can also be viewed as acting on $L^p((0,\tau),B^0)$, $1\le p<\infty$. The following theorem is contained in both \cite[Proposition 9.3.7]{Ha} and \cite[Proposition 6.7]{Lu}. Recall that the fractional Sobolev space $W^{\alpha,p}((0,\tau), B^0)$ is defined by
\[
W^{\alpha,p}((0,\tau), B^0):=(L^p((0,\tau),B^0),W^{1,p}((0,\tau),B^0))_{\alpha,p},\quad (\alpha,p)\in (0,1)\times (1,\infty).
\] 

\begin{theorem}\label{thm5}
For all $(\alpha,p)\in (0,1)\times (1,\infty)$, let $\mathscr{K}_{\alpha,p}:=L^p((0,\tau),D_{\mathbb{A}}(\alpha,p))\cap W^{\alpha,p}((0,\tau), B^0)$. Then it holds $\mathscr{C}\in \mathscr{B}(L^p((0,\tau),B^0),\mathscr{K}_{\alpha,p})$.
\end{theorem}

Next, let $\alpha \in (0,1)$, $p\in [1,\infty)$, and define $\mathscr{I}$ by
\[
\mathscr{I}: f\in D_{\mathbb{A}}(\alpha ,p)\mapsto \mathscr{I}(f)(t)=e^{-t\mathbb{A}}f,\; t\in [0,\tau].
\]

\begin{proposition}\label{pro3}
For all $f\in D_{\mathbb{A}}(\alpha ,p)$, we have $\mathscr{I}(f)\in C([0,\tau], D_{\mathbb{A}}(\alpha ,p))$. In particular, for all $X_\alpha \in J_\alpha$ and $f\in D_{\mathbb{A}}(\alpha ,1)$, we have $\mathscr{I}(f)\in C([0,\tau], X_\alpha)$.
\end{proposition}

\begin{proof}
Let $f\in D_{\mathbb{A}}(\alpha ,p)$, denote $\tau_0>\tau$ be arbitrarily fixed, and let $t\in [0,\tau]$ and $h>0$ such that $t+h\in [0,\tau_0]$. As we can replace $1$ by $\tau_0$ in the proof of \cite[Proposition 2.2.9]{Lun}, using
\[
\mathbf{c}:=\sup_{0\le t\le \tau_0}\|e^{-t\mathbb{A}}\|_{\mathscr{B}(D_{\mathbb{A}}(\alpha ,p))}<\infty,
\]
we obtain
\begin{align*}
\|e^{-(t+h)\mathbb{A}}f-e^{-t\mathbb{A}}f\|_{D_{\mathbb{A}}(\alpha ,p)}&=\|e^{-t\mathbb{A}}(e^{-h\mathbb{A}}f-f)\|_{D_{\mathbb{A}}(\alpha ,p)}
\\
&\le \mathbf{c} \|e^{-h\mathbb{A}}f-f\|_{D_{\mathbb{A}}(\alpha ,p)}.
\end{align*}
This, together with \cite[Proposition 2.2.8]{Lun}, yields
\[
\lim_{h\rightarrow 0}\|e^{-(t+h)\mathbb{A}}f-e^{-t\mathbb{A}}f\|_{D_{\mathbb{A}}(\alpha ,p)}=0.
\]
The proof is complete.
\end{proof}

As a consequence of Proposition \ref{pro3}, we have the following corollary.

\begin{corollary}\label{cor2}
For all $f\in D_{\mathbb{A}}(\alpha ,1)$, we have $\mathscr{I}(f)\in C ([0,\tau],C_b^{\lfloor 2\alpha \rfloor, 2\alpha -\lfloor 2\alpha\rfloor})$ if $\alpha\ne \frac{1}{2}$, and $\mathscr{I}(f)\in C ([0,\tau],C_b^1)$ if $\alpha= \frac{1}{2}$.
\end{corollary}

We now introduce the definition of a mild solution of the following abstract Cauchy problem:
\begin{equation}\label{aeq1}
u'(t)+\mathbb{A}u(t)=f(t),\quad u(0)=u_0.
\end{equation}
Let $f\in L^1((0,\tau),B^0)$ and $u_0\in B^0$. Since for all $w\in B^0$ and $t>0$ we have $\frac{d}{dt}\left[e^{-t\mathbb{A}}w\right]=-\mathbb{A}e^{-t\mathbb{A}}w$, applying Duhamel's formula formally to \eqref{aeq1} gives
\begin{equation}\label{ms1}
u(t)=e^{-t\mathbb{A}}u_0+\int_0^t e^{-(t-s)\mathbb{A}}f(s)ds,\quad t\in [0,\tau].
\end{equation}
Recall that $u$ given by \eqref{ms1} is usually called the mild solution of \eqref{aeq1}. According to \cite[Corollary 4.2.4]{Lun}, $u\in L^\infty ((0,\tau), B^0)\cap C((0,\tau],B^0)$ and satisfies
\begin{equation}\label{ms2}
\|u(t)\|_{B^0}\le \|u_0\|_{B^0}+\|f\|_{L^1((0,\tau),B^0)}, \quad t\in [0,\tau].
\end{equation}
In the case $u_0=0$ and $f\in L^\infty((0,\tau),B^0)$, we have $u=\mathscr{C}(f)$; in the case $f=0$, we have $u=\mathscr{I}(u_0)$.

Given that $e^{-t\mathbb{A}}u_0\in C^\infty ((0,\infty), D(A))$ for all $u_0\in B^0$ (see, for example, \cite[Proposition 2.1.1]{Lun}), as a consequence of Proposition \ref{pro2}, we obtain the following regularity result concerning the mild solution of \eqref{aeq1}.

\begin{theorem}\label{thm3}
For all $0<\alpha<1$, $0<\epsilon <\tau$, $f\in L^\infty((0,\tau),B^0)$ and $u_0\in B^0$, the mild solution $u$ of \eqref{aeq1} belongs to $C^{1-\alpha}([\epsilon,\tau],D_{\mathbb{A}}(\alpha,1))$. In particular, $u\in C^{1-\alpha}([\epsilon,\tau],X_\alpha)$ for all $X_\alpha \in J_\alpha$.
\end{theorem}

We also have the following regularity result, which is an immediate consequence of \cite[Theorem 4.3.1]{Lun}. 

\begin{theorem}\label{thm4}
Let $0<\alpha<1$, $u_0\in B^2$, $f\in C^\alpha([0,\tau],B^0)$ and assume that $Au_0+f(0)\in D_{\mathbb{A}}(\alpha,\infty)$. Then the mild solution $u$ of \eqref{aeq1} belongs to $\mathscr{X}_\alpha:=C^{1+\alpha}([0,\tau],B^0)\cap C^\alpha ([0,\tau], B^2)$. Furthermore, the following inequality holds:
\[
\|u\|_{\mathscr{X}_\alpha}\le \mathbf{c}\left(\|u_0\|_{B^2}+\|Au_0+f(0)\|_{D_{\mathbb{A}}(\alpha,\infty)}+\|f\|_{C^\alpha([0,\tau],B^0)}\right),
\]
where the constant $\mathbf{c}>0$ is independent of $u_0$ and $f$.
\end{theorem}

Note that $\mathscr{X}_\alpha$ defined in Theorem \ref{thm4} satisfies
\[
\mathscr{X}_\alpha\hookrightarrow C^{1+\alpha}([0,\tau],C_b^0)\cap C^\alpha ([0,\tau], C_b^2).
\]

In light of Lemma \ref{lem1}, Theorem \ref{thm4} yields the following corollary.

\begin{corollary}\label{cor3}
Let $p\ge 1$. Under the assumptions and notations of Theorem \ref{thm4}, we have $u\in \mathscr{Y}_{\alpha,p}:=\tilde{C}^1([0,\tau],D_{\mathbb{A}}(\alpha,p))\cap \mathrm{Lip}([0,\tau],D_{\mathbb{A}}(\alpha,\infty))$. In particular, for all $X_\alpha \in J_\alpha$, we have $u\in \mathscr{Z}_\alpha:=\tilde{C}^1([0,\tau],X_\alpha)\cap \mathrm{Lip}([0,\tau],D_{\mathbb{A}}(\alpha,\infty))$. In addition, the following inequality holds:
\[
\|u\|_{H}\le \mathbf{c}\left(\|u_0\|_{B^2}+\|Au_0+f(0)\|_{D_{\mathbb{A}}(\alpha,\infty)}+\|f\|_{C^\alpha([0,\tau],B^0)}\right),\quad H\in \left\{\mathscr{Y}_{\alpha,p},\mathscr{Z}_\alpha\right\},
\]
where the constant $\mathbf{c}>0$ is independent of $u_0$ and $f$.
\end{corollary}

\subsection{Semilinear heat equations}\label{sb3.2}

For completeness, we recall a definition. Let $X$ be a Banach space, $X_0$ a subspace of $X$, and $\mathbf{A}:D(\mathbf{A})\subset X\rightarrow X$. The part of $\mathbf{A}$ in $X_0$, denoted by $\mathbf{A}_0$, is defined as follows:
\[
D(\mathbf{A}_0):=\{f\in D(\mathbf{A});\; \mathbf{A}f\in X_0\},\quad \mathbf{A}_0:D(\mathbf{A}_0)\rightarrow X_0: f\mapsto \mathbf{A}_0f:=\mathbf{A}f.
\]

Let $\alpha\in (0,1)$ and $X_\alpha \hookrightarrow B^0$ satisfying that the part of $\mathbb{A}$ in $X_\alpha$ is sectorial in $X_\alpha$ and $D_{\mathbb{A}}(\alpha,1)\hookrightarrow X_\alpha \hookrightarrow D_{\mathbb{A}}(\alpha,\infty)$. As pointed out in \cite[comments after (7.0.2)]{Lun}, $X_\alpha \in J_\alpha$ ; as examples of such spaces, we can take $X_\alpha=D_{\mathbb{A}}(\alpha ,p)$, $p\in [1,\infty]$, or $X_\alpha=D(\mathbb{A}^\alpha)=B^{2\alpha}$.

Let $\tau >0$ be arbitrarily fixed and let $f:[0,\tau]\times X_\alpha \rightarrow B^0$ be a continuous function satisfying: for all $\rho>0$, there exists a constant $\varkappa_\rho>0$ such that
\begin{align}
&\|f(t,w_1)-f(t,w_2)\|_{B^0}\le\varkappa_\rho \|w_1-w_2\|_{X_\alpha},\label{heat4}
\\
&\hskip 6cm t\in [0,\tau],\; \|w_j\|_{B^0}\le \rho,\; j=1,2.\nonumber
\end{align}
Then consider the following Cauchy problem in $B^0$:
\begin{equation}\label{heat5}
\left\{
\begin{array}{ll}
u'(t)+\mathbb{A}u(t)=f(t,u(t))\quad \mbox{in}\; (0,\tau ),
\\
u(0)=u_0.
\end{array}
\right.
\end{equation}
As in the linear case, the function $u:[0,\tau]\rightarrow B^0$ given by
\[
u(t)=e^{-t\mathbb{A}}u_0+\int_0^t e^{-(t-s)\mathbb{A}}f(s,u(s))ds,\quad t\in [0,\tau],
\]
is usually called the mild solution of the Cauchy problem \eqref{heat5}.

The following local existence and uniqueness result follows from \cite[Theorem 7.1.2]{Lun}. Hereafter, the ball of a Banach space $X$ of radius $\rho>0$ centered at $x\in X$ will be denoted by $\mathbf{B}_X(x,\rho)$.

\begin{theorem}\label{thm6}
For all $u_\ast\in X_\alpha$, there exist $\rho_\ast>0$, $0<\tau_\ast \le \tau$ and $\mathbf{c}>0$, depending only on $u_\ast$, $\varkappa_{\rho_\ast}$ and $\alpha$, such that for all $u_0\in \mathbf{B}_{X_\alpha}(u_\ast,\rho_\ast)$, the Cauchy problem \eqref{heat5} admits a unique mild solution $u=u(\cdot, u_0)\in C^\alpha([0,\tau_\ast],B^0)\cap C([0,\tau_\ast], X_\alpha)$. Furthermore, we have
\begin{equation}\label{heat6}
\|u(t,u_0)-u(t,w_0)\|_{X_\alpha}\le \mathbf{c}\|u_0-w_0\|_{X_\alpha},\quad t\in [0,\tau_\ast],\; u_0,w_0\in \mathbf{B}_{X_\alpha}(u_\ast,\rho_\ast).
\end{equation}
\end{theorem}

\begin{example}\label{exa1}
{\rm
In this example, $0<\alpha <1$ and $X_\alpha =D(\mathbb{A}^\alpha)=B^{2\alpha}$. Let $m\ge 1$ be an integer, $\varphi \in C([0,\tau])$, and define $f(t,w)=\varphi(t) w^m$ for $w\in B^{2\alpha}$. In light of Theorem \ref{thm1} (iv), $f:[0,\tau]\times B^{2\alpha} \rightarrow B^{2\alpha}$ is continuous. Furthermore, for all $\rho>0$, we verify that
\[
\|f(t,w_1)-f(t,w_2)\|_{B^0}\le \varkappa_\rho\|w_1-w_2\|_{B^0},
\quad t\in [0,\tau],\;  \|w_j\|_{B^{2\alpha}}\le \rho,\; j=1,2,
\]
where $\varkappa_\rho=m(2\pi)^{-n(m-1)}\rho^{m-1}\|\varphi\|_{C([0,\tau])}$. Consequently, $f$ satisfies the above conditions, and Theorem \ref{thm6} therefore applies to this $f$.

With the notations of Theorem \ref{thm6}, assume that $u_0\in \mathbf{B}_{B^{2\alpha}}(u_\ast,\rho_\ast)\cap B^2$ is such that $\mathbb{A}u_0\in D_{\mathbb{A}}(\alpha,\infty)$, and $\varphi\in C^\alpha([0,\tau])$ in such a way that $t\mapsto f(t,u(t))$ belongs to $C^\alpha([0,\tau_\ast],B^{2\alpha})$. Applying Theorem \ref{thm4}, we obtain that \[u\in C^{1+\alpha}([0,\tau_\ast],B^0)\cap C^\alpha([0,\tau_\ast], B^2).\]
}
\end{example}

Let
$
I(u_0)=\{s\in [0,\tau];\; \mbox{\eqref{heat5} admits a mild solution}\; u\in L^\infty((0,s),B^0)\}
$
and $\tau(u_0)=\sup I(u_0)$. As proved in \cite[Proposition 7.1.8]{Lun}, if $\tau(u_0)<\tau$, then
\[
\limsup_{t\rightarrow \tau(u_0)}\|f(t,u(t))\|_{B^0}=\infty.
\]
Note that the mild solution of \eqref{heat5} is well defined on $[0,\tau(u_0))$.

The regularity result proved in Example \ref{exa1} is in fact valid in a general context. Precisely, from \cite[Proposition 7.1.10, (iv)]{Lun} we have the following result.

\begin{proposition}\label{pro4}
Assume that $f$ additionally satisfies: for all $u_0\in X_\alpha$, there exist $\varrho >0$ and $\mathbf{c}>0$ such that
\[
\|f(t,w)-f(t,u_0)\|_{B^0}\le \mathbf{c}(t-s)^\alpha ,\quad 0\le s\le t\le \tau,\; w\in \mathbf{B}_{X_\alpha}(u_0,\varrho).
\]
Let $u_0\in B^2$ and suppose that $\mathbb{A}u_0+f(0,u_0)\in D_{\mathbb{A}}(\alpha ,\infty)$. If $u: [0,\tau(u_0))\rightarrow B^0$ is the mild solution of \eqref{heat5}, then for all $0<\tilde{\tau}<\tau(u_0)$ we have $u\in C^{1+\alpha}([0,\tilde{\tau}],B^0)\cap C^\alpha([0,\tilde{\tau}],B^2)$.
\end{proposition}

\subsection{The action of the Gaussian heat kernel on $B^0$} \label{sb3.3}

Let us first recall that the Gaussian heat kernel is given by
\[
G_t(x)=\frac{1}{(4\pi t)^\frac{n}{2}} e^{-\frac{|x|^2}{4t}},\quad x\in \mathbb{R}^n,\; t>0.
\]
Clearly, $G_t\in \mathscr{S}$ for all $t>0$. Therefore, since $B^0\hookrightarrow C_b^0$, $G_t\ast f$ is well defined for all $f\in B^0$. 

In the following, we use the fact that $\hat{G}_t(\xi)=e^{-t|\xi|^2}$ for all $t>0$ and the following property holds. 
\begin{proposition}\label{prohk1}
$\mathrm{(i)}$ For all $f\in B^0$ and $t>0$, it holds $\widehat{G_t\ast f}=\hat{G}_t\hat{f}$.
\\
$\mathrm{(ii)}$ For all $f\in B^0$, $G_t\ast f$ converges to $f$ in $B^0$ as $t\to 0$.
\\
$\mathrm{(iii)}$ For all $f\in B^0$, $s\ge 0$ and $t>0$, $G_t\ast f\in B^s$.
\end{proposition}

\begin{proof}
$\mathrm{(i)}$ If $f\in \mathscr{S}$, then $\widehat{G_t\ast f}=\hat{G}_t\hat{f}$; since $\|\hat{G}_t\|_{L^\infty}\le 1$, we obtain
\begin{equation}\label{hk1}
\|G_t\ast f\|_{B^0}\le \|f\|_{B^0}.
\end{equation}
Let $f\in B^0$ and let $(f_k)$ be a sequence in $\mathscr{S}$ converging to $f$ in $B^0$. In view of \eqref{hk1}, we have
\[
\|G_t\ast f_k-G_t\ast f_\ell\|_{B^0}\le \|f_k-f_\ell\|_{B^0},\quad \mbox{for all}\; k,\ell.
\]
Thus, $(G_t\ast f_k)$ is a Cauchy sequence in $B^0$. Hence, there exists $g\in B^0$ such that $(\widehat{G_t\ast f_k})$ converges to $\hat{g}$ in $L^1$. Since $\widehat{G_t\ast f_k}=\hat{G}_t\hat{f}_k$ for all $k$, and $(\hat{G}_t\hat{f}_k)$ converges to $\hat{G}_t\hat{f}$ in $L^1$, we get $\hat{g}=\hat{G}_t\hat{f}$. On the other hand, since $f_k\to f$ in $C_b^0$, we have $(G_t\ast f_k)\to G_t\ast f$ in $\mathscr{S}'$. The continuity of the Fourier transform on $\mathscr{S}'$ then implies that $(\widehat{G_t\ast f_k})$ converges to $\widehat{G_t\ast f}$ in $\mathscr{S}'$. Hence, $\widehat{G_t\ast f}=\hat{G}_t\hat{f}$. 
\\
$\mathrm{(ii)}$ Let $f\in B^0$. Then we have
\[
\|G_t\ast f-f\|_{B^0}=\|(e^{-t|\xi|^2}-1)\hat{f}\|_{L^1}.
\]
The desired result follows from the dominated convergence theorem. 
\\
$\mathrm{(iii)}$ This follows from the fact that $\langle \xi\rangle^s \hat{G}_t\in L^\infty$ for all $s\ge 0$ and $t>0$.
\\
The proof is complete.
\end{proof}

We already know that $e^{tA}f=G_t\ast f$ for all $f\in \mathscr{S}$. Therefore, $G_t$ is precisely the heat kernel of $e^{tA}$.

\subsection{Backward heat equation in $B^0$}\label{sb3.4}
In this subsection, we revisit the classical backward heat equation and establish a conditional stability result for this inverse problem.
To this end, we consider a simplified form of (\ref{heat1}), which yields the following abstract Cauchy problem:
\begin{equation}\label{eq_BHP}
\left\{
\begin{array}{l}
\partial_t u (t,x)-\Delta u(t,x) = 0, \quad \mbox{in}\; (0,\tau)\times \mathbb{R}^n,
\\
v(0,\cdot)=u_0.
\end{array}
\right.
\end{equation}
Let the final time $\tau >0$ be fixed, and recall that $A=\Delta$. We consider the following initial-state-to-final-state operator
\[
\mathfrak{I}_\tau: u_0\in B^0\mapsto \mathfrak{I}_\tau (u_0):= e^{\tau A} u_0\in B^0.
\]
Define
\[
\mathfrak{R}_\tau:=\{h\in C_b^0;\; e^{\tau|\xi|^2}\hat{h}\in L^1\},
\]
which we equip with the norm
\[
\|h\|_{\mathfrak{R}_\tau}:=\|e^{\tau|\xi|^2}\hat{h}\|_{L^1}.
\]
Note that $\mathfrak{R}_\tau\subset \bigcap_{s\ge 0}B^s$.

\begin{lemma}\label{lemif1}
$\mathfrak{I}_\tau$ is an isometric isomorphism from $B^0$ onto $\mathfrak{R}_\tau$.
\end{lemma}

\begin{proof}
From Proposition \ref{prohk1}, we have
\[
\widehat{\mathfrak{I}_\tau( u_0)}=e^{-\tau |\xi|^2}\widehat{u_0},\quad u_0\in B^0,
\]
and hence
\[
\|\mathfrak{I}_\tau(u_0)\|_{\mathfrak{R}_\tau}=\|e^{\tau |\xi|^2}\widehat{\mathfrak{I}_\tau(u_0)}\|_{L^1}=\|\widehat{u_0}\|_{L^1}=\|u_0\|_{B^0}.
\]
That is, $\mathfrak{I}_\tau$ is an isometry from $B^0$ to $\mathfrak{R}_\tau$. 

Let $g\in \mathfrak{R}_\tau$ and $u_0\in C_b^0$ be such that $\widehat{u_0}:=e^{\tau |\xi|^2}\hat{g}\in L^1$. Applying Proposition \ref{prohk1} again, we obtain
\[
\hat{g}=e^{-\tau |\xi|^2}\widehat{u_0}=\widehat{G_\tau \ast u_0}.
\]
Consequently, $g=G_\tau \ast u_0=\mathfrak{I}_\tau(u_0)$. This shows that $\mathfrak{I}_\tau$ is surjective, and thus completes the proof.
\end{proof}

\begin{proposition}\label{proif1}
Let $s>0$. For all $u_0\in B^s$ and $\rho >0$, we have
\begin{equation}\label{if1}
\|u_0\|_{B^0}\le e^{\tau \rho}\|\mathfrak{I}_\tau(u_0)\|_{B^0}+\rho^{-\frac{s}{2}}\|u_0\|_{B^s}.
\end{equation}
\end{proposition}

\begin{proof}
Let $u_0 \in B^s$ and $g=\mathfrak{I}_\tau(u_0)$. We saw in the preceding proof that $\widehat{u_0}=e^{\tau|\xi|^2}\hat{g}$. Hence,
\begin{equation}\label{if2}
\int_{|\xi|\le \sqrt{\rho}}|\widehat{u_0 }(\xi)|d\xi=\int_{|\xi|\le \sqrt{\rho}}e^{\tau|\xi|^2}|\hat{g}(\xi)|d\xi \le e^{\tau \rho}\|\hat{g}\|_{L^1}=e^{\tau \rho}\|g\|_{B^0},\quad \rho >0.
\end{equation}
On the other hand, we have
\begin{equation}\label{if3}
\int_{|\xi|>\sqrt{\rho}}|\widehat{u_0 }(\xi)|d\xi=\int_{|\xi|>\sqrt{\rho}}\langle \xi\rangle^{-s}[\langle \xi\rangle^s|\widehat{u_0 }(\xi)|]d\xi\le \rho^{-\frac{s}{2}}\|u_0\|_{B^s},\quad \rho>0.
\end{equation}
Combining \eqref{if2} and \eqref{if3}, we obtain \eqref{if1}.
\end{proof}

Since $0<\|\mathfrak{I}_\tau(u_0 )\|_{B^0}<\|u_0 \|_{B^s}$ if $u_0 \ne 0$, minimizing the right-hand side of \eqref{if1} yields the following corollary.

\begin{corollary}\label{corif1}
Let $s>0$. The following conditional stability holds for all $u_0\in B^s\setminus\{0\}$:
\[
\|u_0 \|_{B^0} \le \left(\tau +\frac{s}{2}\right)^{\frac{s}{2}}\left|\ln \frac{\|\mathfrak{I}_\tau(u_0 )\|_{B^0}}{\|u_0 \|_{B^s}}\right|^{-\frac{s}{2}} \|u_0 \|_{B^s}.
\]
\end{corollary}

We now make a few remarks concerning the above corollary. The backward heat equation aims to recover the initial value of the heat equation (\ref{eq_BHP}) from the final-time measurement \(u(x,\tau)\) or \(\mathfrak{I}_\tau (u_0)\). As a time-reversed version of the heat equation, this problem is generally ill-posed, meaning that small perturbations in the measurement data can lead to large deviations in the reconstruction. As observed in the above corollary, under the mild assumption that \(u_0\in B^s\setminus\{0\}\), we obtain a logarithmic stability estimate for \(\|u_0 \|_{B^0}\) in terms of the measurement \(\|\mathfrak{I}_\tau(u_0 )\|_{B^0}\). To the best of our knowledge, this constitutes the first stability estimate for inverse problems within the spectral Barron space setting, and it is consistent with classical results (see, e.g., \cite{Is2006}).

\section{Time-fractional heat equations in $B^0$}\label{s4}

In this section, we extend our discussion to time-fractional heat equations, which model anomalous diffusion processes in which the mean-squared displacement grows sublinearly in time. Such a phenomenon has been observed in many physical, biological, and financial systems. 

\subsection{Definitions and notations}\label{sb4.2}
We begin by introducing the fractional derivative that will be used in this section. In the following, $X$ is an arbitrary Banach space and $\alpha \in (0,1)$. We recall that the modified fractional Riemann–Liouville derivative is formally defined, for $f:[0,\infty)\rightarrow X$, by the formula
\[
f^{(\alpha)}(t):=\frac{1}{\Gamma(1-\alpha)}\left[ \frac{d}{dt}\int_0^t(1-t)^{-\alpha}f(s)ds-t^{-\alpha}f(0)\right].
\]
Here, $\Gamma$ denotes the usual Gamma function.

For all $\tau >0$ and $p\in [1,\infty]$, we have established in \cite[Lemma 1.1]{CY} that if $f\in W^{1,p}((0,\tau),X)$, then $f^{(\alpha)}\in L^p((0,\tau),X)$. Consequently, we define the space $W_{0,}^{(\alpha),p}((0,\tau),X)$ as the closure of
\[
W_{0,}^{1,p}((0,\tau),X):=\left\{f\in W^{1,p}((0,\tau),X);\; f(0)=0\right\}
\]
with respect to the norm
\[
\|f\|_{W_{0,}^{(\alpha),p}((0,\tau),X)}:=\|f\|_{L^p((0,\tau),X)}+\|f^{(\alpha)}\|_{L^p((0,\tau),X)},\quad f\in W_{0,}^{(\alpha),p}((0,\tau),X).
\]
We also define
\[
W_{0,\mathrm{loc}}^{(\alpha),p}((0,\infty),X):=\left\{ f\in L^1_{\mathrm{loc}}((0,\infty),X);\; f\in W_{0,}^{(\alpha),p}((0,\tau),X)\; \mbox{for all}\; \tau >0\right\}.
\]
From \cite[Proposition 1.3]{CY}, if $f\in W_{0,\mathrm{loc}}^{1,p}((0,\infty),X)$, then $f\in C([0,\tau],X)$ for all $\tau>0$, and $f(0)=0$.

Next, we recall the vector-valued Laplace transform. For $f\in L^1((0,\infty),X)$ or $f\in L^\infty((0,\infty),X)$, it is given by the formula
\[
\mathscr{L}(f)(\lambda)=\int_0^\infty e^{-\lambda t}f(t)dt,\quad \lambda\in \mathbb{C}_+:=\{z\in \mathbb{C};\; \Re \lambda >0\},
\]
where the integral is understood in the Bochner sense.

\subsection{Cauchy problems for the time-fractional heat equations}\label{sb4.2}
Let $\alpha\in (0,1)$, $u_0\in B^0$ and consider the Cauchy problem
\begin{equation}\label{fcp1}
u^{(\alpha)}(t)-Au(t)=0,\;  t>0,\quad u(0)=u_0.
\end{equation}
Proceeding as in \cite{CY}, we verify, at least formally, that $v=\mathscr{L}(u)$ is the solution of the equation
\begin{equation}\label{fcp2}
\lambda^\alpha v(\lambda)-Av(\lambda)=\lambda^{\alpha-1}u_0, \quad \lambda\in \mathbb{C}_+.
\end{equation}
That is, we have $v(\lambda)=K(\lambda)u_0$, where
\[
K(\lambda):=\lambda^{-1}(1-\lambda^{-\alpha}A)^{-1},\quad \lambda\in \mathbb{C}_+.
\]
From Proposition \ref{pro1}, we obtain
\[
\|K(\lambda)\|_{\mathscr{B}(B^0)}\le \frac{1}{|\lambda|},\quad \lambda\in \mathbb{C}_+.
\]
We can proceed as in the proof of \cite[Proposition 2.1]{CY} to show that $K: \mathbb{C}_+\rightarrow \mathscr{B}(B^0)$ is holomorphic and
\begin{equation}\label{fcp3}
\|K^{(\ell)}(\lambda)\|_{\mathscr{B}(B^0)}\le \frac{\ell !}{|\lambda|^{\ell +1}},\quad \lambda\in \mathbb{C}_+,\; \ell \in \mathbb{N}_0=\mathbb{N}\cup \{0\}.
\end{equation}

In view of \eqref{fcp3}, examining the calculations in \cite[Subsection 2.1.2]{CY} shows that all the results of that subsection remain valid for the operator $A$ considered here. Consequently, \cite[Theorems 2.3 and 2.4]{CY} apply to our operator $A$ with $E=B^0$. Before stating the obtained results, we define
\[
\mathscr{U}_\alpha:=\left \{ f\in L^\infty((0,\infty),B^2)\cap C^\alpha([0,\infty),B^0));\; f^{(\alpha)}\in  L^\infty((0,\infty),B^0)\right\},
\]
which we equip with the norm
\[
\|f\|_{\mathscr{U}_\alpha}:=\|f\|_{L^\infty((0,\infty),B^2)}+[f]_\alpha+\|f^{(\alpha)}\|_{L^\infty((0,\infty),B^2)}.
\]
In what follows, we define
\[
\kappa: =\frac{1}{\Gamma(\alpha+1)}+1.
\]
\begin{theorem}\label{thmfcp1}
For all $u_0\in B^2$, the Cauchy problem \eqref{fcp1} admits a unique solution $u\in \mathscr{U}_\alpha$ such that
\[
\|u\|_{\mathscr{U}_\alpha}\le \kappa\|u_0\|_{B^2}.
\]
\end{theorem}

Regarding the Cauchy problem
\begin{equation}\label{fcp4}
u^{(\alpha)}(t)-Au(t)=f(t),\;  t>0,\quad u(0)=0,
\end{equation}
we have the following result.

\begin{theorem}\label{thmfcp2}
For all $f\in W_{0,\mathrm{loc}}^{(1-\alpha),1}((0,\infty),B^2)$, the Cauchy problem \eqref{fcp4} admits a unique solution $u\in L_{\mathrm{loc}}^\infty((0,\infty),B^2)\cap C([0,\infty),B^0)$ such that $u^{(\alpha)}\in L_{\mathrm{loc}}^1((0,\infty),B^0)$. Furthermore, the following inequalities hold:
\begin{align*}
&\|u(t)\|_{B^2}\le \kappa \|f^{(1-\alpha)}\|_{L^1((0,t),B^2)},\quad t>0,
\\
&\|u^{(\alpha)}(t)\|_{B^0}\le \kappa \|f^{(1-\alpha)}\|_{L^1((0,t),B^2)}+\|f(t)\|_{B^0},\quad t>0.
\end{align*}
\end{theorem}

In light of \cite[Lemma 2.5]{CY} and its proof, we can state the following H\"older continuity result.

\begin{lemma}\label{lemfcp1}
Let $u$ be the solution of the Cauchy problem \eqref{fcp4} corresponding to $f\in W_{0,\mathrm{loc}}^{(1-\alpha),1}((0,\infty),B^2)$. For all $\tau >0$, we have
\[
\|u(t)-u(s)\|_{B^2}\le \kappa_\tau|t-s|^\alpha \|f^{(1-\alpha)}\|_{L^1((0,\tau),B^2)},\quad t,s\in [0,\tau],
\]
where $\kappa_\tau=(1+\tau^{1-\alpha})\kappa$.
\end{lemma}

\section{Extension to anisotropic elliptic operators}\label{s5}

As in \cite{CLT}, the previous results can be extended to elliptic operators with variable coefficients. With the exception of Subsections \ref{sb3.3} and \ref{sb3.4}, all results obtained for the Laplace operator remain valid, subject to minor modifications, for the elliptic operators with variable coefficients considered in this section.

\subsection{Operators with constant coefficients} \label{sb5.1}
Let $\mathbf{a}^0=(a_{k\ell}^0)$ be a symmetric matrix satisfying: there exists a constant $\sigma >0$ such that
\[
\mathbf{a}^0\xi \cdot \xi \ge \sigma |\xi|^2,\quad \xi \in \mathbb{R}^n.
\]
Let $\mathcal{A}_0:B^0\rightarrow B^0$ be the unbounded operator given by
\[
\mathcal{A}_0:=\mathrm{div}( \mathbf{a}^0\nabla u ),\quad u\in D(\mathcal{A}_0)=B^2.
\]

We proceed as in the case of $A$ to show that $\mathcal{A}_0$ is closed with dense domain.

\begin{proposition}\label{proani1}
$\mathrm{(i)}$ The inclusion $\{ z\in \mathbb{C};\; \Re z>0\} \subset \rho(\mathcal{A}_0)$ holds, and
\begin{equation}\label{reso1}
\|(\lambda -\mathcal{A}_0)^{-1}\|_{\mathscr{B}(B^0)}\le \frac{1}{|\lambda|},\quad \Re z>0; 
\end{equation}
hence $\mathcal{A}_0$ is sectorial.
\\
$\mathrm{(ii)}$ The following inequalities hold:
\begin{align}
&\|(\lambda -\mathcal{A}_0)^{-1}\|_{\mathscr{B}(B^0,B^1)}\le \frac{1}{\sqrt{\sigma |\lambda|}},\quad \Re \lambda\ge \sigma,\label{reso2}
\\
&\|(\lambda -\mathcal{A}_0)^{-1}\|_{\mathscr{B}(B^0,B^2)}\le \frac{1}{\sigma},\quad \Re \lambda\ge \sigma.\label{reso3}
\end{align}
\end{proposition}

\begin{proof}
$\mathrm{(i)}$ For $f\in B^0$ and $\lambda \in \mathbb{C}$ such that $\Re \lambda >0$, consider the equation
\begin{equation}\label{res1}
\lambda u - \mathrm{div}(\mathbf{a}^0\nabla u)=f.
\end{equation}
In the Fourier space, this equation becomes
\[
(\lambda +\mathbf{a}^0\xi \cdot \xi)\hat{u}=\hat{f},
\]
and hence
\[
\hat{u}=\frac{1}{\lambda +\mathbf{a}^0\xi \cdot \xi}\hat{f}.
\]
Since 
\begin{align*}
|\lambda +\mathbf{a}^0\xi \cdot \xi|^2&=(\Re \lambda +\mathbf{a}^0\xi \cdot \xi)^2+(\Im \lambda)^2
\\
&\ge (\Re \lambda +\sigma |\xi|^2)^2+(\Im \lambda)^2
\\
&\ge |\lambda|^2,
\end{align*}
we obtain
\[
|\hat{u}|\le \frac{|\hat{f}|}{|\lambda|}.
\]
On the other hand, since the function $\xi \in \mathbb{R}^n\mapsto \frac{1+|\xi|^2}{\left((\Re \lambda +\sigma |\xi|^2)^2+(\Im \lambda)^2\right)^{\frac{1}{2}}}$ is bounded, we get $u\in B^2$. Therefore, \eqref{res1} admits a unique solution $u=(\lambda-\mathcal{A}_0)^{-1}f\in B^2$ and \eqref{reso1} holds. 

$\mathrm{(ii)}$ Using 
\[
\sup_{\xi \in \mathbb{R}^n}\frac{1+|\xi|^2}{\left((\Re \lambda +\sigma |\xi|^2)^2+(\Im \lambda)^2\right)^{\frac{1}{2}}}\le \frac{1}{\sigma},\quad \Re \lambda\ge \sigma ,
\]
we find
\[
\|u\|_{B^2}\le \frac{1}{\sigma}\|f\|_{B^0},\quad \Re \lambda\ge \sigma,
\]
which, combined with the interpolation inequality \eqref{ii}, yields
\[
\|u\|_{B^1}\le \frac{1}{\sqrt{\sigma |\lambda|}}\|f\|_{B^0},\quad \Re \lambda\ge \sigma.
\]
The proof is complete.
\end{proof}

\subsection{Operators with variables coefficients}\label{sb5.2}

Let $\mathbf{a}_0$ and $\mathcal{A}_0$ be as in the preceding subsection, and let $\mathbf{a}=(a_{k\ell})$ be a matrix with variable coefficients satisfying: $a_{k\ell}\in a_{k\ell}^0+B^1\subset C_b^1$ for all $1\le k,\ell \le n$, and there exists $0<\delta <\sigma$ such that
\begin{align}
&2\sum_{k,\ell}\|a_{k\ell}-a^0_{k\ell}\|_{B^1}\le \delta , \qquad \rm{and }\label{cond1}
\\
&  \left(\sum_{k,\ell}\|a_{k\ell}-a_{k\ell}^0\|_{C_b^0}^2\right)^{\frac{1}{2}}\le \delta .\label{cond2}
\end{align}
We then derive how condition \eqref{cond2} guarantees that $\mathbf{a}$ is uniformly positive definite. Indeed, by the Cauchy–Schwarz inequality, for all $x, \xi \in \mathbb{R}^n$, we have
\[
\sum_{k,\ell}|(a_{k\ell}(x)-a_{k\ell}^0)\xi_k\xi_\ell |\le \left(\sum_{k,\ell}|a_{k\ell}(x)-a_{k\ell}^0|^2\right)^{\frac{1}{2}}\left(\sum_{k,\ell}(\xi_k\xi_\ell )^2\right)^{\frac{1}{2}}
\le \delta |\xi|^2,
\]
and therefore
\[
\mathbf{a}\xi \cdot \xi =\mathbf{a}_0\xi \cdot \xi +(\mathbf{a}-\mathbf{a}^0)\xi \cdot \xi \ge (\sigma -\delta)|\xi|^2.
\]

Define the unbounded operator $\mathcal{A}:B^0\rightarrow B^0$ by
\[
\mathcal{A}:=\mathrm{div}(\mathbf{a} \nabla u),\quad u\in D(\mathcal{A})=B^2.
\]
As shown in the proof of the proposition below, condition \eqref{cond1} will guarantee that $\mathcal{A}$ is sectorial.

\begin{proposition}\label{proani2}
The inclusion $\{z\in \mathbb{C};\; \Re z\ge \sigma\}\subset \rho(\mathcal{A})$ is satisfied, along with the following inequalities:
\begin{align*}
&\|(\lambda -\mathcal{A})^{-1}\|_{\mathscr{B}(B^0)}\le \frac{\sigma}{(\sigma-\delta)|\lambda|},\quad \Re \lambda \ge \sigma,
\\
&\|(\lambda -\mathcal{A})^{-1}\|_{\mathscr{B}(B^0,B^1)}\le \frac{\sqrt{\sigma}}{(\sigma-\lambda)\sqrt{|\lambda|}},\quad \Re \lambda \ge \sigma,
\\
&\|(\lambda -\mathcal{A})^{-1}\|_{\mathscr{B}(B^0,B^2)}\le \frac{1}{\sigma-\delta},\quad \Re \lambda \ge \sigma.
\end{align*}
In particular, $\mathcal{A}$ is sectorial.
\end{proposition}

\begin{proof} 
Let $\lambda \in \mathbb{C}$ with $\Re \lambda \ge \sigma$ and $f\in B^0$. Consider the equation
\begin{equation}\label{res2}
\lambda u - \mathrm{div}(\mathbf{a} \nabla u)=f.
\end{equation}
We rewrite this equation in the form
\begin{equation}\label{res3}
\lambda u - \mathrm{div}(\mathbf{a}^0\nabla u)= \mathrm{div}((\mathbf{a}-\mathbf{a}_0) \nabla u)+f.
\end{equation}
We are therefore led to solve the following equation
\begin{equation}\label{res4}
u-T(\lambda)u=(\lambda-\mathcal{A}_0)^{-1}(f),
\end{equation}
where
\[
T(\lambda):=(\lambda-\mathcal{A}_0)^{-1}B,\quad Bu:=\left(\mathrm{div}((\mathbf{a}-\mathbf{a}_0) \nabla u)\right).
\]
We will show that $\|T\|_{\mathscr{B}(B^2)}<1$, which guarantees the solvability of \eqref{res4} in $B^2$. Starting from the formal inequalities
\[
\|\partial_kg\|_{B^0}\le \|g\|_{B^1},\quad \|\partial_{k\ell}^2g\|_{B^0}\le \|g\|_{B^2},\quad 1\le k,\ell \le n,
\]
we get
\begin{align*}
\|\mathrm{div}((\mathbf{a}-\mathbf{a}_0) \nabla u)\|_{B^0}&\le \sum_{k,\ell}\left[\|\partial_k(a_{k\ell}-a^0_{k\ell})\partial_\ell u\|_{B^0}+\|(a_{k\ell}-a^0_{k\ell})\partial_{k \ell}^2 u\|_{B^0}\right]
\\
&\le \sum_{k,\ell}\left[\|\partial_k(a_{k\ell}-a^0_{k\ell})\|_{B^0}\|\partial_\ell u\|_{B^0}+\|a_{k\ell}-a^0_{k\ell}\|_{B^0}\|\partial_{k \ell}^2 u\|_{B^0}\right]
\\
&\le 2\sum_{k,\ell}\|a_{k\ell}-a^0_{k\ell}\|_{B^1}\|u\|_{B^2}.
\end{align*}
Then, using Proposition \ref{proani1} and condition \eqref{cond1}, we obtain
\begin{align*}
\|T(\lambda)u\|_{B^2}&\le \frac{1}{\sigma}\|\mathrm{div}((\mathbf{a}-\mathbf{a}_0) \nabla u)\|_{B^0}
\\
&\le  \frac{2}{\sigma}\sum_{k,\ell}\|a_{k\ell}-a^0_{k\ell}\|_{B^1}\|u\|_{B^2}
\\
&\le  \frac{\delta}{\sigma}\|u\|_{B^2}.
\end{align*}
Hence, $1-T(\lambda):B^2\rightarrow B^2$ is invertible with $\|(1-T(\lambda))^{-1}\|_{\mathscr{B}(B^2)}\le \frac{\sigma}{\sigma-\delta}$. Consequently, \eqref{res4} admits a unique solution 
\[
u=(1-T(\lambda))^{-1}(\lambda-\mathcal{A}_0)^{-1}(f)\in B^2,
\]
and 
\begin{equation}\label{res5}
\|u\|_{B^2}\le \frac{1}{\sigma-\delta}\|f\|_{B^0}.
\end{equation}
On the other hand, we have
\begin{align*}
\|T(\lambda)u\|_{B^0}&\le \frac{1}{|\lambda|}\|\mathrm{div}((\mathbf{a}-\mathbf{a}_0) \nabla u\|_{B^0} 
\\
&\le \frac{\delta}{|\lambda|} \|u\|_{B^2}
\\
&\le \frac{\delta}{|\lambda|(\sigma-\delta)} \|f\|_{B^0},
\end{align*}
and, since $u=T(\lambda)u+(\lambda-\mathcal{A}_0)^{-1}(f)=(\lambda-\mathcal{A})^{-1}f$, we get
\begin{equation}\label{res6}
\|u\|_{B^0}\le \frac{\sigma}{(\sigma-\delta)|\lambda|}\|f\|_{B^0}.
\end{equation}
Finally, in view of \eqref{ii}, \eqref{res5} and \eqref{res6} give
\[
\|u\|_{B^1}\le\frac{\sqrt{\sigma}}{(\sigma-\lambda)\sqrt{|\lambda|}}\|f\|_{B^0}.
\]
The proof is complete.
\end{proof}

Let $\alpha >0$. We recall that the fractional derivative of order $\frac{\alpha}{2}$ acting on $w\in B^\alpha$ is given by the pseudo-differential operator
\[
(-\Delta)^{\frac{\alpha}{2}} w(x):=\frac{1}{(2\pi)^n}\int_{\mathbb{R}^n} e^{ix\cdot \xi}|\xi|^\alpha \hat{w}(\xi)d\xi.
\]
By density of $\mathscr{S}$ in $B^s$, for all $s\ge 0$, we have
\[
\mathscr{F}((-\Delta)^{\frac{\alpha}{2}} w)=|\xi|^\alpha\hat{w}.
\]
Therefore, $(-\Delta)^{\frac{\alpha}{2}}$ maps $B^{\alpha+s}$ continuously into $B^s$ for all $s\ge 0$.

Let $0<\alpha<2$ and $m\in B^{2-\alpha}$. According to \cite[Proposition 2.3]{CLT}, the operator defined by
\begin{equation}\label{comp}
\mathcal{C}:w\in B^2\mapsto m(-\Delta)^{\frac{\alpha}{2}}w\in B^0
\end{equation}
is compact. Then, applying \cite[Proposition 2.4.3]{Lun}, we obtain the following perturbation result.

\begin{corollary}\label{corani1}
Let $0<\alpha<2$, $m\in B^{2-\alpha}$ and $\mathcal{C}$ be given by \eqref{comp}. Then $\mathcal{A}+\mathcal{C}:B^0\rightarrow B^0$ with domain $D(\mathcal{A}+\mathcal{C})=B^2$ is sectorial.
\end{corollary}

In fact, a more general result follows by combining Theorem \ref{thm1} (iv) and \cite[Proposition 2.4.1]{Lun}. Precisely, we have the following.

\begin{corollary}\label{corani2}
Let $p,q\in B^0$, $0<\alpha <2$ and 
\[
\mathcal{B}: w\in B^\alpha\mapsto pw+q(-\Delta)^{\frac{\alpha}{2}}w. 
\]
Then $\mathcal{A}+\mathcal{B}:B^0\rightarrow B^0$ with domain $D(\mathcal{A}+\mathcal{B})=B^2$ is sectorial.
\end{corollary}

We conclude this subsection with the following proposition.

\begin{proposition}\label{proani3}
For all $\tilde{\sigma}>\sigma$, $\tilde{\sigma}-\mathcal{A}$ is a positive operator.
\end{proposition}

\begin{proof}
From Proposition \ref{proani2}, we have
\[
\|(\lambda -\mathcal{A})^{-1}\|_{\mathscr{B}(B^0)}\le \frac{\sigma}{(\sigma-\delta)|\lambda|},\quad \lambda\in  (\sigma ,\infty).
\]
Replacing $\lambda$ by $-\lambda$, we get
\[
\|(\lambda +\mathcal{A})^{-1}\|_{\mathscr{B}(B^0)}\le \frac{\sigma}{(\sigma-\delta)|\lambda|},\quad \lambda\in (-\infty,-\sigma).
\]
Then, for $\tilde{\sigma}>\sigma$ arbitrarily fixed, we obtain
\[
\|((\lambda +\tilde{\sigma})-(\tilde{\sigma}-\mathcal{A}))^{-1}\|_{\mathscr{B}(B^0)}\le \frac{\sigma}{(\sigma-\delta)|\lambda|},\quad \lambda+\tilde{\sigma}\in (-\infty,\tilde{\sigma}-\sigma).
\]
In particular,
\[
\|((\lambda +\tilde{\sigma})-(\tilde{\sigma}-\mathcal{A}))^{-1}\|_{\mathscr{B}(B^0)}\le \frac{\sigma}{(\sigma-\delta)|\lambda+\tilde{\sigma}-\tilde{\sigma}|},\quad \lambda+\tilde{\sigma}\in (-\infty,0].
\]
Since $ |\lambda+\tilde{\sigma}-\tilde{\sigma}|=|\lambda+\tilde{\sigma}|+\tilde{\sigma}\ge \min(1,\tilde{\sigma})(|\lambda+\tilde{\sigma}|+1)$, we arrive at
\[
\|(\lambda -(\tilde{\sigma}-\mathcal{A}))^{-1}\|_{\mathscr{B}(B^0)}\le \frac{\mathbf{m}}{|\lambda|+1},\quad \lambda \in (-\infty,0],
\]
where 
\[
\mathbf{m}:=\frac{\sigma}{\min(1,\tilde{\sigma})(\sigma-\delta)}.
\]
Thus, by \cite[Definition 4.1]{Lu}, $\tilde{\sigma}-\mathcal{A}$ is a positive operator, as desired.
\end{proof}

\subsection{The heat kernel}\label{sb5.3}

Since $|\lambda| \ge |\lambda-\sigma|$ whenever $\Re \lambda \ge \sigma$, it follows from Proposition \ref{proani2} that
\[
\|(\lambda -\mathcal{A})^{-1}\|_{\mathscr{B}(B^0)}\le \frac{\sigma}{(\sigma-\delta)|\lambda-\sigma|},\quad \Re \lambda \ge \sigma.
\]
Applying \cite[Proposition 2.1.1]{Lun} then gives a constant $\mathbf{c}_0=\mathbf{c}_0(\sigma,\delta)\ge 1$ such that
\[
\|e^{t\mathcal{A}}\|_{\mathscr{B}(B^0)}\le \mathbf{c}_0e^{\sigma t},\quad t\ge 0. 
\]
From \cite[Lemma 2.1.6]{Lun}, we have
\begin{equation}\label{sem1}
(\lambda-\mathcal{A})^{-1}=\int_0^\infty e^{-\lambda t}e^{t\mathcal{A}},\quad \Re \lambda >\sigma . 
\end{equation}
According to \cite[Theorem 2.3.6]{Da}, for all $t>0$, the operator $e^{t\mathcal{A}}$, viewed as an operator on $L^2$, admits a heat kernel $K_t$ satisfying
\[
0\le K_t(x,y)\le \mathbf{c}t^{-\frac{n}{2}}\quad t>0,\;  \mbox{a.e.}\; x,y\in \mathbb{R}^n .
\]
By a heat kernel, we mean a measurable function $K_t:\mathbb{R}^n\times \mathbb{R}^n\rightarrow \mathbb{R}$ such that
\[
e^{t\mathcal{A}}f(x)=\int_{\mathbb{R}^n}K_t(x,y)f(y)dy,\quad t>0,\; \mbox{a.e.}\; x\in \mathbb{R}^n.
\]
The heat kernel of $e^{t\mathcal{A}_0}$ will be denoted by $K_t^0$.

We now show that $K_t^0$ is related to the Gaussian heat kernel
\[
G_t(x)=\frac{1}{(4\pi t)^\frac{n}{2}} e^{-\frac{|x|^2}{4t}},\quad x\in \mathbb{R}^n,\; t>0.
\]
To this end, we start with some formal calculations. Let $u$ be the solution of
\[
\partial_tu(t,x)-\mathrm{div}(\mathbf{a}^0\nabla u(t,x)) = 0,\quad u(0,x)=f(x).
\]
If $v(t,y)=u(t,\sqrt{\mathbf{a}^0}\, y)$, one verifies that $v$ is the solution of the equation
\[
\partial_tv(t,y)-\Delta v(t,y) = 0,\quad u(0,x)=f(\sqrt{\mathbf{a}^0}\, y).
\]
For $f\in \mathscr{S}$, we have
\[
u(t,\sqrt{\mathbf{a}^0}\, y)=\int_{\mathbb{R}^n}K_t^0(\sqrt{\mathbf{a}^0}\, y,z)f(z)dz=\int_{\mathbb{R}^n}G_t(y-z)f(\sqrt{\mathbf{a}^0}\, z)dz.
\]
A change of variables in the last integral yields
\[
\int_{\mathbb{R}^n}K_t^0(\sqrt{\mathbf{a}^0}\, y,z)f(z)dz=\frac{1}{\mathrm{det}\left(\sqrt{\mathbf{a}^0}\right)}\int_{\mathbb{R}^n}G_t\left(y-\sqrt{\mathbf{a}^0}^{-1}z\right)f(z)dz.
\]
Consequently, we get
\[
K_t^0(y,z)=\frac{1}{\mathrm{det}\left(\sqrt{\mathbf{a}^0}\right)}G_t\left(\sqrt{\mathbf{a}^0}^{-1}(y-z)\right).
\]
In view of this identity, Proposition \ref{prohk1} remains valid with $G_t$ replaced by $K_t^0$.

Next, we establish a relationship between $K_t^0$ and $K_t$. Let $t>0$ and $x,y\in \mathbb{R}^n$. Examining the results of \cite[Section 2]{CK}, we verify that
\begin{equation}\label{kii}
K_t(x,y)=K^0_t(x,y)+\int_0^t\int_{\mathbb{R}^n}K_{t-s}^0(x,z)H_s(z,y)dzds,
\end{equation}
where $H_t(x,y)$ is given by the series
\[
H_t(x,y)=\sum_{j=1}^\infty H_t^j(x,y),
\]
with $H_t^1(x,y)=\mathrm{div}_x(\mathbf{a}(x)\nabla_xK_t^0(x,y))$ and
\[
H_t^{j+1}(x,y)=\int_0^t\int_{\mathbb{R}^n}H_{t-s}^1(x,z)H_s^j(z,y)dyds.
\]
Furthermore, the following estimate holds:
\[
|H_t(x,y)|\le \mathbf{c}t^{-\frac{n+1}{2}}e^{\mathbf{c}t}e^{-\frac{|x-y|^2}{\mathbf{c}t}}.
\]

We also mention that the kernel $K_t$ satisfies a two-sided Gaussian estimate. Indeed, from \cite[Corollary 3.2.8 and Theorem 3.3.4]{Da}, there exist constants $\mathfrak{c}_0=\mathfrak{c}_0(\mathbf{a}_0,\sigma,\delta)>0$ and $\mathfrak{c}_1=\mathfrak{c}_1(\mathbf{a}_0,\sigma,\delta)>0$ such that
\begin{equation}\label{kin}
\mathfrak{c}_0t^{-\frac{n}{2}}e^{-\frac{|x-y|^2}{\mathfrak{c}_0t}}\le K_t(x,y)\le \mathfrak{c}_1t^{-\frac{n}{2}}e^{-\frac{|x-y|^2}{\mathfrak{c}_1t}},\quad t>0,\; x,y\in \mathbb{R}^n.
\end{equation}
We point out that a two-sided Gaussian inequality for a parabolic equation with time-dependent coefficients was established in \cite{FS}.

As a final result of this section, we show that $K_t$ is the heat kernel of $e^{t\mathcal{A}}$. However, it should be noted that it seems difficult to use \eqref{kii} to transfer Proposition \ref{prohk1}, in which $G_t$ is replaced by $K_t^0$, from $K_t^0$ to $K_t$.

\begin{proposition}\label{proani4}
It holds
\begin{equation}\label{hkr}
e^{t\mathcal{A}}f(x)= \int_{\mathbb{R}^n}K_t(x,y)f(y)dy,\quad f\in B^0,\quad t>0,\; f\in B^0.
\end{equation}
Furthermore, if $f\in B^0$ satisfies $f\ge 0$, then $e^{t\mathcal{A}}f\ge 0$ for all $t\ge 0$.
\end{proposition}

\begin{proof}
For all $f\in \mathscr{S}$, we have 
\[
e^{t\mathcal{A}}f(x)=\int_{\mathbb{R}^n}K_t(x,y)f(y)dy.
\]
By \eqref{kin}, for all $t>0$, $K_t\in L^\infty(\mathbb{R}_x^n,L^1(\mathbb{R}^n_y))$. Hence,
\begin{align*}
\|e^{t\mathcal{A}}f\|_{L^\infty}&\le \|K_t\|_{L^\infty(\mathbb{R}_x^n,L^1(\mathbb{R}^n_y))}\|f\|_{L^\infty}
\\
&\le \|K_t\|_{L^\infty(\mathbb{R}_x^n,L^1(\mathbb{R}^n_y))}\|f\|_{B^0}.
\end{align*}
Then the density of $\mathscr{S}$ in $B^0$ establishes the validity of \eqref{hkr}. The second assertion follows immediately from the fact that $K_t$ is nonnegative for all $t>0$.
\end{proof}

\bigskip	\noindent {\bf Data Availability Statement.} Data sharing not applicable to this article as no datasets were generated or analysed during the current study.

\bigskip	\noindent {\bf Conflict of Interest.} The authors have no conflicts of interest to declare that are relevant to the content of this article.

\bibliographystyle{amsplain}
\bibliography{refs}

%
%
%
%
%
%
%
%
%
%
%


\end{document}